\documentclass{amsart}
\usepackage{amsmath, amssymb, amsthm}
\usepackage[margin=1in]{geometry}
\usepackage[pdftex]{hyperref}

\makeatletter

\@addtoreset{equation}{section}
\makeatother
  
\newcommand{\supp}{\mathrm{supp}\hspace{0.5mm}}
\newcommand{\op}{\mathrm{Op}}

\newcommand{\R}{\mathbb{R}}
\newcommand{\Z}{\mathbb{Z}}
\newcommand{\calF}{\mathcal{F}}
\newcommand{\calS}{\mathcal{S}}
\newcommand{\aand}{\quad\mathrm{and}\quad}
\newcommand{\iif}{\quad\textrm{if}\quad}
\newcommand{\ffor}{\quad\textrm{for}\quad}

\newtheorem{theorem}{Theorem}[section]
\newtheorem{proposition}[theorem]{Proposition}
\newtheorem{corollary}[theorem]{Corollary}
\newtheorem{lemma}[theorem]{Lemma}
\newtheorem*{theoremA}{Theorem A}

\theoremstyle{definition}
\newtheorem{remark}[theorem]{Remark}

\begin{document}

\title[Pseudodifferential operators on $\alpha$-modulation spaces]
{Pseudodifferential operators with symbols in the H\"ormander class
$S^0_{\alpha,\alpha}$ on $\alpha$-modulation spaces}

\author[T. Kato]{Tomoya Kato}
\author[N. Tomita]{Naohito Tomita}

\address{Department of Mathematics, Graduate School of Science, Osaka University, Toyonaka, Osaka 560-0043, Japan}
\email[Tomoya Kato]{t.katou@cr.math.sci.osaka-u.ac.jp}
\email[Naohito Tomita]{tomita@math.sci.osaka-u.ac.jp}


\begin{abstract}
In this paper, we study the boundedness of pseudodifferential operators
with symbols in the H\"ormander class $S^0_{\rho,\rho}$
on $\alpha$-modulation spaces $M_{p,q}^{s,\alpha}$,
and consider the relation between $\alpha$ and $\rho$.
In particular, 
we show that pseudodifferential operators with symbols in $S^0_{\alpha,\alpha}$
are bounded on all $\alpha$-modulation spaces $M^{s,\alpha}_{p,q}$,
for arbitrary $s\in\R$
and for the whole range of exponents $0 < p,q \leq \infty$.
\end{abstract}

\maketitle

\section{Introduction} \label{sec1}

In Gr\"obner's Ph.D. thesis \cite{grobner 1992}, $\alpha$-modulation spaces $M_{p,q}^{s,\alpha}$ were introduced as intermediate spaces
between modulation spaces $M_{p,q}^s$ and (inhomogeneous) Besov spaces $B^s_{p,q}$.
The parameter $\alpha \in [0,1)$ determines how the frequency space is decomposed. 
Modulation spaces which are constructed
by the uniform frequency decomposition correspond to the case $\alpha = 0$
and Besov spaces which are constructed by the dyadic decomposition
can be regarded as the limiting case $\alpha \to 1$.
See the next section for the precise definition of $\alpha$-modulation spaces.

Let $b \in \R$, $0 \le \delta \le \rho \le 1$,
$\delta<1$, $1<p,q<\infty$ and $s \in \R$.
It is known that
all operators of class $\op(S^b_{\rho,\delta})$
are bounded on $L^p(\R^n)$ if and only if $b \le -|1/p-1/2|(1-\rho)n$
(see \cite[Chapter VII, Section 5.12]{stein 1993}),
and the same condition assures the $B^s_{p,q}$-boundedness,
namely the boundedness of operators of class
$\op(S^b_{\rho,\delta})$, $b \leq -|1/p-1/2|(1-\rho)n$,
on $B^s_{p,q}(\R^n)$ holds
(see, e.g., Bourdaud \cite{bourdaud 1982}, Gibbons \cite{gibbons 1978}
and Sugimoto \cite{sugimoto 1988}).
It should be remarked that
the boundedness of operators of class $\op(S^0_{1,1})$
on $B^s_{p,q}$ also holds for $s>0$ (see the references above).
On the other hand, as a difference between boundedness
on Besov and modulation spaces,
it is known that all operators of class $\op(S^0_{0,0})$
are bounded on $M^s_{p,q}$
(see, e.g., Gr\"ochenig and Heil \cite{grochenig heil 1999},
Tachizawa \cite{tachizawa 1994} and Toft \cite{toft 2004}).
Moreover, Sugimoto and Tomita \cite[Theorem 2.1]{sugimoto tomita PAMS 2008}
proved that the boundedness of operators of class
$\op(S^0_{1,\delta})$, $0<\delta<1$, on $M^0_{p,q}$, $q \neq 2$,
does not hold in general,
and also

\begin{theoremA}[{\cite[Theorem 1]{sugimoto tomita JFAA 2008}}]
Let $1 < q < \infty$, $b \in \R$,
$0 \leq \delta \leq \rho \leq 1$ and $ \delta < 1$.
Then, all pseudodifferential operators with symbols in
$S_{\rho, \delta}^b$ are bounded on $M_{2,q}^{0}(\R^n)$
if and only if $b \leq - | 1/q - 1/2 | \delta n$.
\end{theoremA}

In this paper, we discuss the $M_{p,q}^{s,\alpha}$-boundedness of
pseudodifferential operators with symbols in the so-called
exotic class $S^0_{\rho, \rho}$,
and try to clarify the relation between $\alpha$ and $\rho$.
For the boundedness of pseudodifferential operators,
we will use the following terminology with a slight abuse.
Let $s,t \in \R$. 
If there exist a constant $C_\sigma$ such that the estimate
\begin{equation*}
\left\| \sigma(X,D) f \right\|_{M_{p,q}^{t,\alpha}} 
\leq
C_\sigma
\left\| f \right\|_{M_{p,q}^{s,\alpha} }
\ffor
f \in \calS (\R^n)
\end{equation*}
holds, 
then we simply say that
the pseudodifferential operator $\sigma(X,D)$ 
is bounded from
$M_{p,q}^{s,\alpha} $ to $M_{p,q}^{t,\alpha}$.
In particular, if $s=t$, then we say that 
$\sigma(X,D)$ 
is bounded on $M_{p,q}^{s,\alpha}$.
If $p,q < \infty$, the boundedness mentioned above can be extended to the formal one by density.
Borup \cite{borup 2003} proved that
all pseudodifferential operators of class $\op(S_{\alpha, \alpha}^0)$ are bounded
from $M_{p,q}^{s,\alpha}$ to $M_{p,q}^{s-(1-\alpha),\alpha}$
for the space dimension $n=1$.
Borup and Nielsen \cite{borup nielsen AM 2006, borup nielsen 2008} also obtained
the boundedness of operators of class $\op(S_{\alpha, 0}^0)$
on $M_{p,q}^{s,\alpha}$.
Our purpose is to improve the result of \cite{borup 2003}
by removing the loss of the smoothness $1-\alpha$
and that of \cite{borup nielsen AM 2006, borup nielsen 2008}
by replacing $\delta=0$ with $\delta=\alpha$
in the full range $0<p,q \le \infty$.
Our main result is the following:

\begin{theorem} \label{main theorem}
Let $0 \leq \alpha < 1$, $ 0 < p , q \leq \infty$ and $s \in \R $.
Then, all pseudodifferential operators with symbols in $S_{\alpha, \alpha}^0$
are bounded on $M_{p,q}^{s,\alpha} (\R^n)$.
\end{theorem}

\begin{remark}
More precisely, we will prove that
there exists a positive integer $N$ such that the estimate
\begin{equation} \label{main theorem estimate}
\left\| \sigma(X,D) f \right\|_{M_{p,q}^{s,\alpha}} 
\lesssim
\left\|\sigma ; S^0_{\alpha,\alpha}\right\|_N
\left\|  f  \right\|_{M_{p,q}^{s,\alpha} } 
\end{equation}
holds for all $\sigma \in S_{\alpha, \alpha}^0$ 
and all $f \in \calS(\R^n)$.
\end{remark}

Recalling the relation $S^0_{\rho, \delta_1} \subset S^{0}_{\rho, \delta_2}$ 
for $\delta_1 \leq \delta _2$, 
we see that the class $S_{\alpha, \alpha}^0 $ in Theorem \ref{main theorem} is wider 
than $S_{\alpha, 0}^0$ in \cite{borup nielsen AM 2006}. 
Furthermore, we recall the well-known result 
by Calder\'on and Vaillancourt \cite{calderon vaillancourt 1972},
where it was stated that for any $0 \leq \alpha < 1$, 
$\op (S^0_{\alpha, \alpha}) \subset \mathcal{L} (L^2)$.
Since we have $M^{0,\alpha}_{2,2} = L^2$ for any $0 \leq \alpha < 1 $,
we see that Theorem \ref{main theorem} contains their result.

Also, more generally, we have the following.

\begin{corollary} \label{main cor}
Let $0 \leq \alpha < 1$, $ 0 < p , q \leq \infty$, $s,b \in \R $
and $0 \leq \delta \leq \alpha \leq \rho \leq 1$. 
Then, all pseudodifferential operators with symbols in $S_{\rho, \delta}^b$ are bounded from
$M_{p,q}^{s,\alpha} (\R^n)$ to $M_{p,q}^{s-b,\alpha} (\R^n)$.
\end{corollary}

As a direct consequence of Theorem \ref{main theorem}, Theorem A 
and inclusion relations between modulation and $\alpha$-modulation spaces, 
we immediately have the following statement.

\begin{corollary} \label{sharp cor}
Let $1 < q < \infty$, $q\neq 2$, $s \in \R$, $0 \leq \delta \leq \rho \leq 1$, 
$0 \leq \delta, \alpha < 1$ and $\alpha \leq \rho$.
Then, all pseudodifferential operators with symbols in $S_{\rho, \delta}^0$
are bounded on $M_{2,q}^{s,\alpha} (\R^n)$
if and only if $\delta \leq \alpha$. 
\end{corollary}

To conclude the overview of our results, we comment on the optimality of the symbol class
in Theorem \ref{main theorem}. 
Corollary \ref{sharp cor} implies that 
$\op (S_{\alpha + \varepsilon, \alpha + \varepsilon}^0) 
\not \subset 
\mathcal{L}(M_{2,q}^{s,\alpha})$, 
$q \neq 2$, for any $0 < \varepsilon < 1 - \alpha$. 
On the other hand, 
$\op (S_{\alpha - \varepsilon, \alpha - \varepsilon}^0) 
\not \subset 
\mathcal{L}(M_{p,q}^{s,\alpha})$,
$0 < p < 1$,
for any $0 < \varepsilon < \alpha$ 
(see Remark \ref{alpha -}). 
Therefore, the class $S_{\alpha, \alpha}^0$ in Theorem \ref{main theorem} 
seems to be optimal to obtain the $M_{p,q}^{s, \alpha}$-boundedness.

The plan of this paper is as follows. In Section \ref{sec2}, we will state basic notations 
which will be used throughout this paper, 
and then introduce the definition and some basic properties of $\alpha$-modulation spaces. 
After stating and proving some lemmas needed to show the main theorem in Section \ref{sec3}, 
we will actually prove it in Section \ref{sec4}.

We end this section by mentioning a remark 
on arguments to give a proof of the boundedness. 
If we prove estimate \eqref{main theorem estimate}
for all Schwartz functions $\sigma$ 
on $ \R^{2n} $, then the same estimate holds 
for all $\sigma \in S_{\alpha, \alpha}^0$ by a limiting argument
(see, e.g., the beginning of the proof of \cite[Chapter VII, Section 2.5, Theorem 2]{stein 1993}).
Hence, in the following statements, we will prove Theorem \ref{main theorem} 
for symbols $\sigma$ belonging to $\calS (\R^{2n})$.



\section{Preliminaries}\label{sec2}

\subsection{Basic notations} \label{sec2.1}
In this section, we collect notations which will be used throughout this paper.
We denote by $\R$, $\Z$ and $\Z_+$
the sets of reals, integers and non-negative integers, respectively. 
The notation $a \lesssim b$ means $a \leq C b$ with a constant $C > 0$ 
which may be different in each occasion, 
and $a \sim b $ means $a \lesssim b$ and $b \lesssim a$. 
We will write $\langle x \rangle = (1 + | x |)$ for $x \in \R^n$ 
and $[s] = \max\{ n \in \Z : n \leq s \}$ for $s \in \R$.
We will also use the notation $A = A(\alpha) = \frac{ \alpha }{ 1 - \alpha }$ 
(especially, in Sections \ref{sec3} and \ref{sec4}).

We denote the Schwartz space of rapidly decreasing smooth functions
on $\R^n$ by $\calS = \calS (\R^n)$ and its dual,
the space of tempered distributions, by $\calS^\prime = \calS^\prime(\R^n)$. 
The Fourier transform of $ f \in \calS (\R^n) $ is given by
\begin{equation*}
\calF f  (\xi) = \widehat {f} (\xi) = \int_{\R^n}  e^{-i \xi \cdot x } f(x) d x,
\end{equation*}
and the inverse Fourier transform of $ f \in \calS (\R^n) $ is given by
\begin{equation*}
\calF^{-1} f (x) = \check {f} (x) = (2\pi)^{-n} \int_{\R^n}  e^{i x \cdot \xi } f( \xi ) d\xi.
\end{equation*}

We next recall the symbol class $S_{\rho, \delta}^b = S_{\rho, \delta}^b ( \R^n \times \R^n ) $ 
for $b \in \R$ and $0 \leq \delta \leq \rho \leq 1$, 
which consists of all functions $\sigma \in C^\infty (\R^n \times \R^n)$ satisfying
\begin{equation*}
| \partial_x^\beta \partial_\xi^\gamma \sigma (x , \xi ) | 
\leq C_{\beta, \gamma} \langle \xi \rangle^{b + \delta | \beta | - \rho |\gamma|}
\end{equation*}
for all multi-indices $\beta, \gamma \in \Z_+^n$,
and set
\begin{equation*}
\left\|\sigma ; S^b_{\rho,\delta}\right\|_{N}
=\max_{|\beta|+|\gamma| \leq N}
\left(\sup_{(x, \xi) \in \R^n \times \R^n}
\langle \xi \rangle^{-(b + \delta | \beta | - \rho |\gamma|)}
| \partial_x^\beta \partial_\xi^\gamma \sigma (x , \xi ) |\right)
\end{equation*}
for $N \in \Z_+$.
Note that $S^{b_1}_{\rho_1, \delta_1} \subset S^{b_2}_{\rho_2, \delta_2}$ holds 
if $b_1 \leq b_2$, $\rho_1 \geq \rho_2$ and $\delta_1 \leq \delta _2$. 
For $\sigma \in S_{\rho, \delta}^b$, 
the pseudodifferential operator $\sigma(X,D)$ is defined by
\begin{equation*}
\sigma (X,D) f (x) 
=(2\pi)^{-n} \int_{\R^n} e^{i x \cdot \xi} \sigma (x,\xi) \widehat f (\xi) d\xi
\end{equation*}
for $ f \in \calS (\R^n) $.
We denote by $\op(S^b_{\rho,\delta})$ the class of all pseudodifferential operators 
with symbols in $S^b_{\rho,\delta}$. 
For the case $0\leq \delta <1$ and $\delta \leq \rho \leq 1$, we know 
the statement of the composition rule about the class $S^b_{\rho,\delta}$ called the symbolic calculus.
That is, if $\sigma \in S^b_{\rho,\delta}$ and $\tau \in S^c_{\rho,\delta}$, 
then there exists a symbol $\theta \in S^{b+c}_{\rho,\delta}$ 
satisfying that $\theta (X,D) = \sigma (X,D) \circ \tau (X,D)$.
Moreover, $\theta$ can be chosen so that
\begin{equation*}
\left\| \theta ; S^{b+c}_{\rho,\delta}\right\|_{N}
\lesssim \left\| \sigma ; S^b_{\rho,\delta}\right\|_{M} \cdot \left\| \tau ; S^c_{\rho,\delta}\right\|_{M},
\end{equation*}
where $M$ depends on $N$, $b$, $c$, $\delta$ and $\rho$.
See Stein \cite[Chapter VII, Section 5.8]{stein 1993}. 
The estimate for the symbols just above can be found in Kumano-go \cite[Lemma 2.4]{kumanogo 1981}.

For $m \in L^\infty (\R^n)$, we write the associated Fourier multiplier operator as
\begin{equation*}
m(D) f = \calF^{-1} \left[m \cdot \calF f \right]
\end{equation*}
for $ f \in \calS (\R^n) $,
and especially the Bessel potential 
as $(I - \Delta )^{s/2} f = \calF^{-1} \left[(1 + | \cdot |^2)^{s/2} \cdot \calF f \right]$ 
for $ f \in \calS (\R^n) $ and $s \in \R$. 

In the following, we recall the definitions and properties of function spaces which we will use.
The Lebesgue space $L^p = L^p(\R^n) $ is equipped with the (quasi)-norm 
\begin{equation*}
\| f \|_{L^p} = \left( \int_{\R^n} \big| f(x) \big|^p dx \right)^{1/p} 
\end{equation*}
for $0 < p < \infty$. 
If $p = \infty$, $\| f \|_{\infty} = \textrm{ess}\sup_{x\in\R^n} |f(x)|$. 
Moreover, for a compact subset $\Omega \subset \R^n$, 
$L^p_\Omega 
= L^p_\Omega(\R^n) 
= \{ f \in L^p (\R^n) \cap \calS^\prime (\R^n) : \supp ( \calF f ) \subset \Omega \}$. 
For $0 < q \leq \infty$,
we denote by $\ell^q$ the set of all complex number sequences $\{ a_k \}_{ k \in \Z^n }$ 
such that
\begin{equation*}
\| \{ a_k \}_{ k\in\Z^n } \|_{ \ell^q } 
=
\left( \sum_{ k \in \Z^n }  | a_k  |^q \right)^{ 1 / q } < \infty,
\end{equation*}
if $q < \infty$, and 
$\| \{ a_k \}_{ k\in\Z^n } \|_{ \ell^\infty } = \sup_{k \in \Z^n} |a_k| < \infty$ 
if $q = \infty$.
For the sake of simplicity, we will write $ \| a_k  \|_{ \ell^q } $
instead of the more correct notation $ \| \{ a_k \}_{ k\in\Z^n } \|_{ \ell^q } $.
For a function space $X$, 
we denote by $\mathcal{L}(X)$ the space of all bounded linear operators on $X$. 
We end this subsection with stating the following lemmas from \cite{triebel 1983}.

\begin{proposition} [{\cite[Section 1.5.3]{triebel 1983}}]
\label{convolution 0<p<1}
Let $0 < p \leq 1$. Then we have
\begin{equation*}
\| f \ast g \|_{L^p} 
\leq 
C R^{n(1/p-1)} \| f \|_{L^p} \| g \|_{L^p}
\end{equation*}
for all $f,g \in L^p_\Omega$,
where $\Omega = \{ x \in \R^n : | x - x_0 | \leq R \}$ 
and the constant $C > 0$ is independent of $x_0$ and $R$. 
\end{proposition}

\begin{proposition}[{\cite[Theorem 1.4.1 (i) and Theorem 1.6.2]{triebel 1983}}] 
\label{maximal inequality}
Let $0 < p \leq \infty$. If $0 < r < p$, then we have
\begin{equation*}
\left\| \sup_{y \in \R^n } \frac{ | f (x-y) | }{ 1 + | R y |^{n/r} } \right\|_{L^p(\R^n_x ) } 
\leq
C \| f \|_{L^p}
\end{equation*}
for all $f \in L^p_\Omega$, 
where $\Omega = \{ x \in \R^n : | x - x_0 | \leq R \}$ 
and the constant $C>0$ is independent of $x_0$ and $R$.
\end{proposition}

\subsection{$\alpha$-modulation spaces} \label{sec2.2}
We give the definition of $\alpha$-modulation spaces and their basic properties. 
Let 
$C > 1 $ be a constant 
which depends on the space dimension and $\alpha \in [0,1)$. 
Suppose that a sequence of Schwartz functions $\{ \eta_k^\alpha \}_{k\in\Z^n}$ satisfies that
\begin{itemize} 
\item 
	$\supp \eta_k^\alpha \subset
	\left\{ \xi \in \R^n : 
	\left| \xi - \langle k \rangle^{\alpha/(1-\alpha)}k \right| \leq C \langle k \rangle^{\alpha/(1-\alpha)} 	
	\right\}$;

\item 
	$\left| \partial^\beta_\xi \eta_k^\alpha (\xi) \right| 
	\leq C'_\beta \langle k \rangle^{-|\beta| \alpha/(1-\alpha)}$ 
for every multi-index $\beta \in \Z^n_+$;

\item 
	$\displaystyle{\sum_{k\in\Z^n}} \eta_k^\alpha (\xi) = 1$
for any $\xi \in \R^n$.

\end{itemize} 
Then, for $0< p, q \leq \infty$, $s\in\R$, and $\alpha\in[0,1)$, 
we denote the $\alpha$-modulation space $M_{p,q}^{s,\alpha} $ by
\begin{equation*}
M_{p,q}^{s,\alpha} (\R^n) 
= 
\left\{ 
f\in{\calS}^\prime(\R^n) : 
\left\| f \right\|_{M^{s, \alpha}_{p,q}} 
= 
\left\| \langle k \rangle^{s / (1-\alpha)} 
\left\| \eta_k^\alpha(D) f \right\|_{L^p} 
\right\|_{\ell^q(\Z^n_k)}
< +\infty \right\}.
\end{equation*}
See Borup and Nielsen \cite{borup nielsen JMAA 2006, borup nielsen AM 2006} 
for the abstract definition including the end point case $\alpha = 1$.

We remark that $M_{p,q}^{s,\alpha}$ is a quasi-Banach space (a Banach space if $1 \leq  p,q \leq \infty$)
and $\calS \subset M_{p,q}^{s,\alpha} \subset \calS^\prime$.
In particular, $\calS$ is dense in $M_{p,q}^{s,\alpha}$ for $0 < p,q < \infty$
(see Borup and Nielsen \cite{borup nielsen MN 2006}). 
Moreover, if we choose different decompositions satisfying the conditions above,
they determine equivalent (quasi)-norms of $\alpha$-modulation spaces,
so that the definition of $\alpha$-modulation spaces is 
independent of the choice of the sequence $\{ \eta_k^\alpha \}_{k\in\Z^n}$. 
Next, we recall some basic properties of $\alpha$-modulation spaces.

\begin{proposition} \label{lift op}
Let $0 < p,q \leq \infty$, $s, t \in\R$ and $0\leq \alpha < 1$. Then the mapping $ (I-\Delta)^{t/2} : M_{p,q}^{s,\alpha} \hookrightarrow M_{p,q}^{s-t,\alpha}$ is isomorphic.
\end{proposition}
The proof of Proposition \ref{lift op} is similar to that for Besov spaces in \cite[Section 2.3.8]{triebel 1983}. 
One can find the explicit proof in \cite[Appendix A]{kato 2017}.

\begin{proposition} [{\cite[Proposition 6.1]{han wang 2014}}]
\label{equivalent norm 0}
Let $0 < p,q \leq \infty$, $s \in \R$ and $0\leq\alpha<1$. 
Let a smooth radial bump function $\varrho$ satisfy that $\varrho (\xi) = 1$ on $|\xi| < 1$, 
and $\varrho (\xi) = 0$ on $|\xi| \geq 2$. 
Then, we have
\begin{equation*}
\| f \|_{M_{p,q}^{s,\alpha}} 
\sim 
\left\| \langle k \rangle^{s/(1-\alpha)} \left\| \varrho_k^\alpha(D) f \right\|_{L^p} \right\|_{\ell^q (\Z^n_k)}
\end{equation*} 
for all $f \in M_{p,q}^{s,\alpha}$,
where
\begin{equation*}
\varrho_k^\alpha(\xi) 
= 
\varrho \left( \frac{\xi - \langle k \rangle^{ \alpha/(1-\alpha)} k}{C\langle k \rangle^{ \alpha/(1-\alpha)}} \right).
\end{equation*}
Here, the constant $C > 1$ is the same as in the definition of the sequence $\{ \eta_k^\alpha \}_{k\in\Z^n}$.
\end{proposition}

\begin{lemma} \label{maximal inequality for alpha modulation}
Let $0 < p \leq \infty$ and $0 \leq \alpha < 1 $. If $0 < r < p$, then we have
\begin{equation} \label{max ineq alpha}
\left\| \sup_{y \in \R^n} 
\frac{ \left| \left[ \eta_k^\alpha(D) f \right] (x-y) \right| }
		{ 1 + \left( \langle k \rangle^{ \alpha/(1-\alpha)} | y | \right)^{n/r} } \right\|_{L^p(\R^n_x)} 
\lesssim 
\left\| \eta_k^\alpha(D) f  \right\|_{L^p}
\end{equation}
for all $f \in \calS^\prime (\R^n) $ and all $k \in \Z^n$.
\end{lemma}

\begin{proof}[Proof of Lemma \ref{maximal inequality for alpha modulation}] 
It follows from the definition of the decomposition $\{ \eta_k^\alpha \}_{k\in\Z^n}$ that 
\begin{equation*}
\supp \calF\left[ \eta_k^\alpha(D) f \right] 
\subset 
\left\{ 
\xi \in \R^n
: 
\left| \xi - \langle k \rangle^{ \alpha/(1-\alpha) } k \right| \leq C \langle k \rangle^{ \alpha/(1-\alpha) }
\right\},
\end{equation*} 
so that Lemma \ref{maximal inequality for alpha modulation} follows 
from Proposition \ref{maximal inequality}.
\end{proof}

\begin{remark}
Taking the $\ell^q(\Z^n_k)$ (quasi)-norm of both sides of \eqref{max ineq alpha}, 
we have for $0 \leq \alpha < 1 $, $0 < p,q \leq \infty$ and $0 < r < p$
\begin{equation*}
\left\| \left\| \sup_{y \in \R^n} \frac{ \left| \left[ \eta_k^\alpha(D) f \right] (x-y) \right| }{  1 + \left( \langle k \rangle^{ \alpha/(1-\alpha) } | y | \right)^{n/r} } \right\|_{L^p(\R^n_x)} \right\|_{\ell^q(\Z^n_k)} \lesssim \left\| f  \right\|_{ M_{p,q}^{0,\alpha} }.
\end{equation*}
\end{remark}

We end this subsection by stating the definition of modulation spaces,
which arise as special $\alpha$-modulation spaces for the choice $\alpha=0$.
Another definition and basic properties of modulation spaces can be found 
in \cite{feichtinger 1983, grochenig 2001, kobayashi 2006, kobayashi 2007, wang hudzik 2007}.
Let a sequence of Schwartz functions $\{ \varphi_k \}_{k\in\Z^n}$ satisfy that

\begin{equation*}
\supp \varphi \subset 
\left\{ 	\xi \in \R^n : | \xi | \leq \sqrt{ n } \right\} 
\aand
\sum_{k\in\Z^n} \varphi_k (\xi) = 1 
\textrm{ for any } 
\xi \in \R^n,
\end{equation*}
where $\varphi_k = \varphi(\cdot - k ) $. Then, for $0 < p,q \leq \infty$ and $s \in \R$, we denote the modulation space $M_{ p , q }^s$ by
\begin{equation*}
M_{p,q}^s (\R^n) 
=
\left\{ 
f\in{\calS}^\prime(\R^n) : 
\left\| f \right\|_{M^{s}_{p,q}} 
= 
\Big\| \langle k \rangle^{s} \left\| \varphi_k(D) f \right\|_{L^p} \Big\|_{\ell^q(\Z^n_k)} < +\infty 
\right\}.
\end{equation*}
We finally note that $M_{ p , q }^{s,0} = M_{ p , q }^s$.



\section{Lemmas}\label{sec3}

As stated in the beginning of Section \ref{sec2},
we will use the notation $A = \frac{ \alpha }{ 1 - \alpha }$ 
in the remainder of the paper.

In this section, we prepare some lemmas to use 
in the proof of Theorem \ref{main theorem}. 
As mentioned in the end of Section \ref{sec1}, 
we may assume $\sigma \in \calS(\R^{n}\times\R^{n})$ 
in the following statements. 
We remark that for the partition of unity $\{ \varphi_\ell \}_{\ell \in \Z^n}$ 
used to construct modulation spaces,
it holds that
\begin{equation*}
\sum_{\ell \in \Z^n} \varphi_\ell \left( \xi \right) 
=1 
\textrm{ for any } 
\xi \in \R^n 
\quad\Longrightarrow\quad
\sum_{\ell \in \Z^n} \varphi_\ell \left( \frac{\xi}{\langle m \rangle^A } \right) 
=
1 
\textrm{ for any } 
\xi \in \R^n \textrm{ and }  m \in \Z^n.
\end{equation*}
Then, we can decompose the symbols $\sigma$ as
\begin{align*}
\sigma(x,\xi)
= \sum_{ m \in \Z^n} \sigma(x,\xi) \cdot \eta_m^\alpha (\xi)
= \sum_{ \ell,m \in \Z^n} 
\left( \varphi_\ell \left(
	\frac{D_x}{\langle m \rangle^A }
\right) \sigma \right) (x,\xi)
\cdot 
\eta_m^\alpha (\xi) ,
\end{align*}
where $\{ \eta_m^\alpha \}_{m \in \Z^n}$ is the partition of unity
used for defining the $\alpha$-modulation spaces. 
Put
\begin{equation} \label{def of sigma}
\sigma_{\ell,m}(x,\xi) 
= 
\left( \varphi_\ell \left( \frac{D_x}{\langle m \rangle^A } \right) \sigma \right) (x,\xi) 
\cdot 
\eta_m^\alpha (\xi).
\end{equation}
Then, we have $\sigma_{\ell,m} \in \calS(\R^{n}\times\R^{n})$
and also by Proposition \ref{equivalent norm 0}
\begin{align}
\label{teinei}
\begin{split}
\left\| \sigma(X,D) f \right\|_{M_{p,q}^{s,\alpha}} 
&\sim 
\left\| \langle k \rangle^{\frac{s}{1-\alpha}} 
\left\| \varrho_k^\alpha(D) \left[ \sigma(X,D) f \right] \right \|_{L^p} 
\right\|_{\ell^q (\Z^n_k)}
\\
&=  
\left\| \langle k \rangle^{\frac{ s }{1-\alpha}} 
\left\| \sum_{\ell,m \in \Z^n} \varrho_k^\alpha(D) \left[ \sigma_{\ell,m} (X,D) f \right] \right\|_{L^p} 
\right\|_{\ell^q (\Z^n_k)}.
\end{split}
\end{align} 
In the following, we investigate some properties of $\varrho_k^\alpha(D) \left[ \sigma_{\ell,m} (X,D) f \right]$. 

We first determine the relations among $k,\ell,m \in \Z^n$
by considering the support of $ \calF [\sigma_{\ell,m} (X,D) f ] $.

\begin{lemma} \label{region}
It holds that
\begin{equation*}
\supp \calF [\sigma_{\ell,m} (X,D) f ] 
\subset 
\left\{ \zeta \in \R^n : 
\left| \zeta - \langle m \rangle^A (\ell + m ) \right| \leq (C+ \sqrt{ n } )\langle m \rangle^A 
\right\}
\end{equation*}
for all $f \in \calS (\R^n)$ and all $\ell,m \in \Z^n$.
Here, $C$ is the constant in the definition of the functions $\{ \eta_k^\alpha \}_{k\in\Z^n}$.

Furthermore, it holds that
\begin{equation*}
\supp \varrho_k^\alpha \cap \supp \calF \left[ \sigma_{\ell,m} (X,D) f \right] \neq \varnothing 
\quad\Longrightarrow \quad
| k - m | \lesssim \langle \ell \rangle
\end{equation*}
for all $f \in \calS (\R^n)$ and all $k,\ell,m \in \Z^n$.
\end{lemma}

\begin{remark} \label{cor region}
Lemma \ref{region} implies that 
$\varrho_k^\alpha(D) \left[ \sigma_{\ell,m} (X,D) f \right] $ always vanishes 
unless $ | k - m | \lesssim \langle \ell \rangle $ is satisfied.
\end{remark}

Before beginning with the proof of Lemma \ref{region}, we prepare one lemma.

\begin{lemma} \label{k and m}
It holds that
\begin{equation*}
(\langle k \rangle ^A + \langle m \rangle^A ) | k - m | 
\lesssim \left| \langle k \rangle^A k  - \langle m \rangle^A m \right|
\end{equation*}
for all $k,m \in \Z^n$.
\end{lemma}

\begin{proof}[Proof of Lemma \ref{k and m}] 
By the symmetry of $k, m \in \Z^n$ in the desired inequality,
we may assume that $|k| \ge |m|$,
and divide the argument into three cases.

\noindent {Case 1:} $|k| \geq 2|m|$. Since
\begin{equation*}
(\langle k \rangle ^A + \langle m \rangle^A ) | k - m | 
\lesssim \langle k \rangle ^A | k |
\end{equation*}
and
\begin{equation*}
\left| \langle k \rangle^A k  - \langle m \rangle^A m \right| 
\geq \langle k \rangle^A |k| - \left( 1 + \frac{ | k | }{2} \right)^A \cdot \frac{ | k | }{2}
\gtrsim \langle k \rangle^A | k |,
\end{equation*}
we have $
(\langle k \rangle ^A + \langle m \rangle^A ) | k - m | \lesssim \left| \langle k \rangle^A k  - \langle m \rangle^A m \right| $.

\noindent {Case 2:} 
$|k|=|m|$. Obviously, $(\langle k \rangle ^A + \langle m \rangle^A ) | k - m | = 2 | \langle k \rangle ^A k - \langle m \rangle^A m | $.

\noindent {Case 3:} 
$|m| < |k| < 2|m|$ ($\Rightarrow 1 < \langle k \rangle / \langle m \rangle < 2$). 
Note that
\begin{equation} \label{k m equiv}
| k - m | \leq \left| \frac{ \langle k \rangle^A }{ \langle m \rangle^A } k - m \right|
\end{equation}
holds in this case.
In fact, if $A=0$ ($\Leftrightarrow \alpha=0$), 
then \eqref{k m equiv} holds obviously true. 
Assume that $0 < A < \infty$ ($\Leftrightarrow 0 < \alpha < 1$). 
Since $\langle k \rangle^A / \langle m \rangle^A > 1 $, 
we have the following equivalences
by squaring both sides of the just above estimate and
by rewriting the euclidean norm $|x|^2 = x \cdot x$
in terms of the standard inner product on $\R^n$:
\begin{eqnarray*}
&& 
| k - m | \leq \left| \frac{ \langle k \rangle^A }{ \langle m \rangle^A }k -  m \right| \\
&\Longleftrightarrow& 
2 \left( \frac{ \langle k \rangle^A }{ \langle m \rangle ^A } -1 \right) k \cdot m 
\leq \left( \frac{ \langle k \rangle^{2A} }{ \langle m \rangle ^{2A} } -1 \right ) | k |^2 
\\
&\Longleftrightarrow& 
2 k \cdot m 
\leq \left( \frac{ \langle k \rangle^{A} }{ \langle m \rangle ^{A} } + 1 \right) | k |^2.
\end{eqnarray*}
The last statement is justified from the facts 
$k \cdot m < | k |^2$
and 
$2 < \frac{ \langle k \rangle^{A} }{ \langle m \rangle ^{A} } + 1 $. 
Therefore, it follows that
\begin{align*}
(\langle k \rangle^A + \langle m \rangle^A ) | k - m | & \sim \langle m \rangle^A | k - m | 
\leq \langle m \rangle^A\left| \frac{ \langle k \rangle^A }{ \langle m \rangle^A } k - m \right| 
= | \langle k \rangle ^A k - \langle m \rangle^A m | .
\end{align*} 

Gathering all the cases, we obtain the desired estimate.
\end{proof}

Now, we start the proof of Lemma \ref{region}.

\begin{proof}[Proof of Lemma \ref{region}]
We first consider the former part, 
that is, the support of $\calF \left[ \sigma_{\ell,m} (X,D) f \right] $.
By the Fubini--Tonelli theorem and the definition of $\sigma_{\ell,m}$ in \eqref{def of sigma}, 
we have
\begin{align*}
\calF[\sigma_{\ell,m} (X,D) f ] (\zeta) 
& = 
(2\pi)^{-n} 
\int_{\R^n_\xi} \widehat f (\xi) 
\int_{\R^n_x} e^{- i x \cdot ( \zeta - \xi ) } 
\sigma_{\ell,m} (x,\xi) 
dx d\xi 
\\
& = 
(2\pi)^{-n} \int_{\R^n_\xi} 
\eta_m^\alpha(\xi) \cdot \widehat f (\xi) 
\int_{\R^n_x} e^{ - i x \cdot ( \zeta - \xi ) } 
\left( \varphi_\ell \left( \frac{D_x}{\langle m \rangle^A } \right) \sigma \right) (x,\xi) 
dx d\xi 
\\
&= 
(2\pi)^{-n} 
\int_{\R^n_\xi} 
\eta_m^\alpha(\xi) \cdot \varphi 
\left( \frac{ \zeta - \xi }{\langle m \rangle^A } - \ell \right) 
\cdot \big( \calF_x \sigma \big) ( \zeta - \xi , \xi) 
\cdot \widehat f (\xi) 
d\xi,
\end{align*}
where $\calF_x \sigma$ is the partial Fourier transform of $\sigma(x,\xi)$ 
with respect to the $x$-variable. Hence, the facts
\begin{align*} 
\supp \eta_m^\alpha 
&\subset 
\left\{ \xi \in \R^n : 
\left| \xi - \langle m \rangle^A m \right| \leq C \langle m \rangle^A 
\right\} ; 
\\
\supp \varphi \left( \frac{ \cdot }{\langle m \rangle^A } - \ell \right) 
& \subset 
\left\{ \xi \in \R^n : 
\left| \xi - \langle m \rangle^A \ell \right| \leq \sqrt{ n } \, \langle m \rangle^A 
\right\} 
\end{align*}
yield that 
\begin{equation} \label{support of 2nd}
\supp \calF [\sigma_{\ell,m} (X,D) f ] 
\subset 
\left\{ \zeta \in \R^n : 
\left| \zeta - \langle m \rangle^A (\ell + m ) \right| \leq (C+ \sqrt{ n } )\langle m \rangle^A 
\right\}.
\end{equation}
This is the former part of this lemma.
 
Next, we consider the latter part. 
Assume that 
$
\supp \varrho_k^\alpha 
\cap 
\supp \calF \left[ \sigma_{\ell,m} (X,D) f \right] 
\neq 
\varnothing,
$
and recall from Proposition \ref{equivalent norm 0} that
\begin{equation} 
\label{support of 1st}
\supp \varrho_k^\alpha 
\subset 
\left\{ \zeta \in \R^n : 
\left| \zeta - \langle k \rangle^A k \right| \leq 2C\langle k \rangle^A 
\right\}.
\end{equation}
Then, combining \eqref{support of 2nd} with \eqref{support of 1st}, we obtain
\begin{equation*}
\left| \langle k \rangle^A k  - \langle m \rangle^A (\ell + m ) \right| 
\lesssim 
\langle m \rangle^A + \langle k \rangle^A ,
\end{equation*}
which implies that
\begin{equation*}
\left| \langle k \rangle^A k  - \langle m \rangle^A m \right| \lesssim (\langle m \rangle^A + \langle k \rangle^A ) \langle \ell \rangle
.
\end{equation*}
Hence, we conclude from Lemma \ref{k and m}
\begin{equation*}
| k - m | 
\lesssim \frac{1}{\langle k \rangle ^A + \langle m \rangle^A } 
	\cdot \left| \langle k \rangle^A k  - \langle m \rangle^A m \right| 
\lesssim \langle \ell \rangle,
\end{equation*}
which completes the proof.
\end{proof}

We next prove that $\sigma_{\ell,m} (X,D) f$ has high decay rate with respect to $| \ell |$. 
For technical purposes to prove our main theorem, 
we slightly change the formulation of $\sigma_{\ell,m} (X,D) f $ as follows. 
Choose a function $\kappa \in \calS(\R^n)$ satisfying 
that $ \kappa (\xi) = 1$ on $|\xi| \leq 1$ and $\kappa (\xi) = 0$ on $|\xi| \geq 2$, 
and set
\begin{equation*}
\kappa_m^\alpha(\xi) = \kappa \left( \frac{\xi - \langle m \rangle^A m }{ C\langle m \rangle^A } \right)
\end{equation*}
with the constant $C>1$ 
as in the definition of $\alpha$-modulation spaces (see Section \ref{sec2.2}). 
Then, $\kappa_m^\alpha = 1$ on the support of $\eta_m^\alpha$ 
and thus
\begin{align} 
\label{change to tilde sigma}
\begin{split}
[ \sigma_{\ell,m} (X,D) f ] (x) 
&= 
(2\pi)^{-n} \int_{\R^n_\xi} e^{i x \cdot \xi} 
\left( \varphi_\ell \left( \frac{D_x}{\langle m \rangle^A } \right) \sigma \right) (x,\xi) 
\cdot \eta_m^\alpha (\xi) \cdot \widehat f (\xi) d\xi 
\\
&= 
(2\pi)^{-n} \int_{\R^n_\xi} e^{i x \cdot \xi} 
\left( \varphi_\ell \left( \frac{D_x}{\langle m \rangle^A } \right) \sigma \right) (x,\xi) 
\cdot \eta_m^\alpha (\xi) \kappa_m^\alpha (\xi) \cdot \widehat f (\xi) d\xi 
\\
&= 
[ \widetilde \sigma_{\ell,m} (X,D) \eta_m^\alpha(D) f ](x) ,
\end{split}
\end{align}
where 
$
\widetilde \sigma_{\ell,m} (x,\xi) 
= 
\left( \varphi_\ell \left( \frac{D_x}{\langle m \rangle^A } \right) \sigma \right) (x,\xi) 
\cdot 
\kappa_m^\alpha (\xi)$. 
For the symbol $\widetilde \sigma_{\ell,m}$, we have the following lemma.

\begin{lemma} \label{L^p to L^p}
Let $0 < p \leq \infty$. 
For an arbitrary integer $N \in \Z_+$,
there exists a constant $N^\prime \in \Z_+$ such that
\begin{equation*}
\left\| \widetilde \sigma_{\ell,m} (X,D) \eta_m^\alpha(D) f \right\|_{L^p} 
\lesssim 
\langle \ell \rangle^{-N} 
\| \sigma ; S^{0}_{\alpha, \alpha} \|_{N^\prime} 
\cdot
\left\| \eta_m^\alpha(D) f  \right\|_{L^p}
\end{equation*}
holds for all $\sigma \in S_{\alpha, \alpha}^0 (\R^n \times \R^n)$, 
all $f \in \calS (\R^n)$
and all $\ell,m \in \Z^n$.
\end{lemma}

Before starting the proof, we prepare one lemma. 
In the first step of the proof of Lemma \ref{L^p to L^p},
we will use the following estimate:
\[
\left| [ \widetilde \sigma_{\ell,m} (X,D) \eta_m^\alpha(D) f ](x) \right| 
\leq (2\pi)^{-n} 
\int_{\R^n_y} \left| \eta_m^\alpha(D) f (y) \right| 
\cdot 
\left| \int_{\R^n_\xi} e^{i ( x - y ) \cdot \xi} \widetilde \sigma_{\ell,m} (x,\xi) d\xi \right| 
dy ,
\]
which is justified by the Fubini--Tonelli theorem 
and the fact $\widetilde \sigma_{\ell,m} \in \calS(\R^{n}\times\R^{n})$.
In view of this estimate, the following lemma will be helpful.

\begin{lemma} \label{decay of l and x-y}
For arbitrary integers $M,N \in \Z_+$,
we have
\begin{equation*} 
\left| \int_{\R^n_\xi} e^{i y \cdot \xi} \widetilde \sigma_{\ell,m}(x,\xi) d\xi \right| 
\lesssim 
\langle \ell \rangle^{-N}
\, \| \sigma ; S^{0}_{\alpha, \alpha} \|_{M+N}
\cdot \frac{  \langle m \rangle^{An} }{\left(1+\langle m \rangle^A |y| \right)^{M} }
\end{equation*}
for all $\sigma \in S_{\alpha, \alpha}^0 (\R^n \times \R^n)$,
all $x,y \in \R^n$ 
and all $\ell, m \in \Z^n$.
\end{lemma}

\begin{proof}[Proof of Lemma \ref{decay of l and x-y}] 

In order to obtain the decay with respect to $|y|$, 
we will use integration by parts with respect to the $\xi$-variable, 
so that we first consider the derivatives of $\widetilde \sigma_{\ell,m}$.
For any multi-index $\gamma \in \Z_+^n$ 
with $| \gamma | = M$ and $N \in \Z_+$, 
we have
\begin{equation*}
\left| \partial_\xi^\gamma \left( \widetilde \sigma_{\ell,m} (x,\xi) \right) \right| 
\lesssim 
\langle \ell \rangle ^{-N} \langle m \rangle^{-A M } 
\, \| \sigma ; S^{0}_{\alpha, \alpha} \|_{M+N}
\cdot \chi_{\{ \xi\in\R^n: | \xi - \langle m \rangle^A m | \leq 2C \langle m \rangle^A \} }(\xi),
\end{equation*}
where $\chi_\Omega$ is the characteristic function on the set $\Omega$. 
In fact, the Leibniz rule yields that
\begin{equation} \label{derivative of sigma}
\left| \partial_\xi^\gamma \left( \widetilde \sigma_{\ell,m} (x,\xi) \right) \right| 
\lesssim 
\sum_{\beta \leq \gamma} 
\left| \left( \varphi_\ell \left( \frac{D_x}{\langle m \rangle^A } \right)
	\left(\partial_\xi^\beta \sigma \right)\right) (x,\xi) \right| 
\cdot 
\left| \left( \partial_{\xi}^{\gamma - \beta}\kappa \right) 
	\left(\frac{\xi - \langle m \rangle^A m}{ C \langle m \rangle^A } \right) \right| 
\cdot 
\langle m \rangle^{ - A | \gamma - \beta | }.
\end{equation}
Let us fix $\beta \leq \gamma$ for the moment.
Next, note that $\langle \xi \rangle \sim \langle m \rangle^{ \frac{ 1 }{ 1 - \alpha } }$ 
if 
$\xi \in 
\supp \left( \partial_{\xi}^{\gamma - \beta}\kappa \right) 
\left(\frac{\cdot - \langle m \rangle^A m}{ C \langle m \rangle^A } \right)$ 
in \eqref{derivative of sigma}. 
Then, using the Fubini--Tonelli theorem, we have
\begin{align*}
\left| \left( \varphi_\ell \left( \frac{D_x}{\langle m \rangle^A } \right) 
\left(\partial_\xi^\beta \sigma \right)\right) (x,\xi) \right| 
&=
(2\pi)^{-n} 
\left| \int_{\R^n_z } \left(\partial_\xi^\beta \sigma \right) (z , \xi ) 
\int_{\R^n_{\zeta}} e^{i ( x - z ) \cdot \zeta } \varphi \left( \frac{\zeta}{\langle m \rangle^A } -\ell \right) 
d\zeta dz \right| 
\\
&=  
(2\pi)^{-n} \langle m \rangle^{An} 
\left| \int_{\R^n_z } \left(\partial_\xi^\beta \sigma \right) (z , \xi ) 
\int_{\R^n_{\zeta}} e^{i \langle m \rangle^A ( x - z ) \cdot ( \zeta + \ell ) } \varphi \left( \zeta \right) 
d\zeta dz \right| 
\\
&= 
\langle m \rangle^{An} 
\left| 
\int_{\R^n_z } e^{ - i \langle m \rangle^A \ell \cdot z } \left(\partial_\xi^\beta \sigma \right) (z , \xi ) 
\cdot \check \varphi \left( \langle m \rangle^A ( x - z ) \right) 
dz \right| , 
\end{align*}
where, in the second identity, we used the changes of variables: 
$\zeta' = \zeta / \langle m \rangle^A - \ell $.
If $\ell \neq 0$, 
we apply an $N$-fold integration by parts with respect to the $z$-variable 
to obtain
\begin{align*}
& 
\left| \int_{\R^n_z } e^{ - i \langle m \rangle^A \ell \cdot z } 
\left(\partial_\xi^\beta \sigma \right) (z , \xi )
\cdot \check \varphi \left( \langle m \rangle^A ( x - z ) \right) 
dz \right| 
\\
&\lesssim 
\left( \langle m \rangle^A | \ell | \right)^{-N} 
\sum_{ \substack{ \widetilde \beta \leq \widetilde \gamma \\ | \widetilde \gamma | = N } } 
\langle m \rangle^{ A | \widetilde \gamma - \widetilde \beta | } 
\int_{\R^n_z } 
\left| \left( \partial_{z}^{\widetilde \beta} \partial_\xi^\beta \sigma \right) (z , \xi ) \right| 
\cdot 
\left| \big(\partial_z^{ \widetilde \gamma - \widetilde \beta} \check \varphi \big) 
	\big( \langle m \rangle^A ( x - z ) \big) \right|
dz
\\
&\lesssim
\left( \langle m \rangle^A | \ell | \right)^{-N}
\sum_{ \substack{ \widetilde \beta \leq \widetilde \gamma \\ |\widetilde \gamma | = N } } 
\langle m \rangle^{ A | \widetilde \gamma - \widetilde \beta | }
\cdot \langle m \rangle^{A|\widetilde \beta| - A | \beta | } 
\, \| \sigma ; S^{0}_{\alpha, \alpha} \|_{M+N}
\cdot \langle m \rangle^{-An}
\\
&\sim 
\langle \ell \rangle^{-N} \langle m \rangle^{-A|\beta|-An} 
\, \| \sigma ; S^{0}_{\alpha, \alpha} \|_{M+N}.
\end{align*}
Here, we used 
$ \langle \xi \rangle^\alpha \sim \langle m \rangle^{ A }$ 
in the second inequality, 
and the facts $| \widetilde \beta | + | \widetilde\gamma -  \widetilde\beta | = | \widetilde\gamma| =N$ 
and $ | \ell | \sim \langle \ell \rangle$ for $| \ell | \geq 1$ in the last equivalence. 
On the other hand, if $\ell = 0$, then it similarly follows that
\begin{equation*}
\left| \int_{\R^n_z } e^{ - i \langle m \rangle^A \ell \cdot z } 
\left( \partial_\xi^\beta \sigma \right) (z , \xi ) 
\cdot \check \varphi \left( \langle m \rangle^A ( x - z ) \right) 
dz \right| 
\lesssim
\langle m \rangle^{-A|\beta|-An} 
\, \| \sigma ; S^{0}_{\alpha, \alpha} \|_{M}.
\end{equation*}
Hence, we obtain
\begin{equation} \label{derivative phi}
\left| \left( \varphi_\ell \left( \frac{D_x}{\langle m \rangle^A } \right) 
	\left(\partial_\xi^\beta \sigma \right) \right) (x,\xi) 
\right| 
\lesssim
\langle \ell \rangle^{-N} \langle m \rangle^{-A|\beta|} 
\, \| \sigma ; S^{0}_{\alpha, \alpha} \|_{M+N}
\end{equation}
for all $\ell \in\Z^n$.
Substituting \eqref{derivative phi} into \eqref{derivative of sigma}, 
we have
\begin{align*}
\left| \partial_\xi^\gamma \left( \widetilde \sigma_{\ell,m} (x,\xi) \right) \right| 
&\lesssim 
\sum_{\beta \leq \gamma} \langle \ell \rangle^{-N} \langle m \rangle^{-A|\beta|}
\, \| \sigma ; S^{0}_{\alpha, \alpha} \|_{M+N}
\cdot 
\left| \left( \partial_{\xi}^{\gamma - \beta}\kappa \right) 
	\left(\frac{\xi - \langle m \rangle^A m}{ C \langle m \rangle^A } \right) \right| 
\cdot \langle m \rangle^{ - A | \gamma - \beta | } 
\\
&\lesssim
\langle \ell \rangle^{-N} \langle m \rangle^{-A M}
\,
\| \sigma ; S^{0}_{\alpha, \alpha} \|_{M+N}
\cdot 
\chi_{\{ \xi\in\R^n: | \xi - \langle m \rangle^A m | \leq 2C \langle m \rangle^A \} } (\xi),
\end{align*}
where we used the identity 
$|\beta|+ |\gamma - \beta | = | \gamma | = M$ 
to obtain the last inequality. 
This concludes the result in this step.

Next, we actually investigate the decay of the given integral with respect to $|y|$ 
and obtain the desired estimate for arbitrary $M , N \in \Z_+$. 
Obviously, using the conclusion in Step 1 with $\gamma=0$, we have
\begin{align} 
\label{y=0}
\begin{split}
\left| \int_{\R^n_\xi} e^{i  y \cdot \xi} \widetilde \sigma_{\ell,m}(x,\xi)  d\xi \right| 
&\lesssim 
\langle \ell \rangle^{-N} 
\,
\| \sigma ; S^{0}_{\alpha, \alpha} \|_{N}
\cdot 
\int_{\R^n_\xi} \chi_{\{ \xi\in\R^n: | \xi - \langle m \rangle^A m | \leq 2C \langle m \rangle^A \} } (\xi) d\xi 
\\
& \lesssim
\langle \ell \rangle^{-N} \langle m \rangle^{An} 
\,
\| \sigma ; S^{0}_{\alpha, \alpha} \|_{N}
.
\end{split}
\end{align}
On the other hand, by an $M$-fold integration by parts 
with respect to the $\xi$-variable, we get
\begin{align} 
\label{y neq 0}
\begin{split}
\left| 
	\int_{\R^n_\xi} e^{i  y \cdot \xi} \widetilde \sigma_{\ell,m}(x,\xi) d\xi 
\right| 
&\leq 
| y |^{-M} \sum_{|\gamma|=M}
\int_{\R^n_\xi} 
	\left| \partial_\xi^\gamma \left( \widetilde \sigma_{\ell,m} (x,\xi) \right) \right| d\xi 
\\
&\lesssim
| y |^{-M} 
\langle \ell \rangle^{-N} 
\langle m \rangle^{-AM} 
\,
\| \sigma ; S^{0}_{\alpha, \alpha} \|_{M+N}
\int_{\R^n_\xi} \chi_{\{ \xi\in\R^n: | \xi - \langle m \rangle^A m | \leq 2C \langle m \rangle^A \} } (\xi) d\xi
\\
&\sim 
\langle \ell \rangle^{-N} 
\langle m \rangle^{An}
\,
\| \sigma ; S^{0}_{\alpha, \alpha} \|_{M+N}
\left(\langle m \rangle^{A} | y |\right)^{-M} 
\end{split}
\end{align}
for $y \neq 0$. Combining the conclusions \eqref{y=0} and \eqref{y neq 0}, we obtain
\begin{equation*} 
\left| \int_{\R^n_\xi} e^{i  y \cdot \xi} \widetilde \sigma_{\ell,m}(x,\xi)  d\xi \right| 
\lesssim 
\langle \ell \rangle ^{-N}
\, 
\| \sigma ; S^{0}_{\alpha, \alpha} \|_{M+N}
\cdot
\frac{  \langle m \rangle^{An} }
{ \left( 1 + \langle m \rangle^A | y | \right)^{M} }
\end{equation*}
for all $x, y \in \R^n$ and all $\ell, m \in \Z^n$.
\end{proof}

We are now in a position to prove Lemma \ref{L^p to L^p}.

\begin{proof}[Proof of Lemma \ref{L^p to L^p}] 
We choose $M = ( n + 1) + [n/r] + 1$ for $0 < r < p \leq \infty$ in Lemma \ref{decay of l and x-y}.
Then we have by the Fubini--Tonelli theorem and Lemma \ref{decay of l and x-y}
\begin{align*}
\left| [ \widetilde \sigma_{\ell,m} (X,D) \eta_m^\alpha(D) f ](x) \right| 
&\leq 
(2\pi)^{-n} \int_{\R^n_y} \left| \eta_m^\alpha(D) f (y) \right| \cdot \left| \int_{\R^n_\xi} e^{i ( x - y ) \cdot \xi} \widetilde \sigma_{\ell,m} (x,\xi) d\xi \right| dy 
\\
&\lesssim 
\langle \ell \rangle ^{-N} 
\,
\| \sigma ; S^{0}_{\alpha, \alpha} \|_{N'}
\int_{\R^n_y} \left|  \eta_m^\alpha(D) f  (y) \right| 
\cdot 
\frac{ \langle m \rangle^{An} }{ \left( 1 + \langle m \rangle^A | x-y | \right)^{( n + 1) + [n/r] +1} } dy 
\\
&\lesssim 
\langle \ell \rangle ^{-N} 
\,
\| \sigma ; S^{0}_{\alpha, \alpha} \|_{N'}
\cdot 
\sup_{y\in\R^n}
	\frac{ \left| \eta_m^\alpha(D) f (x-y) \right|  }{ 1 + \left( \langle m \rangle^A | y | \right)^{ n/r} } 
\cdot
\int_{\R^n_y} \frac{ \langle m \rangle^{An} }{ \left( 1 + \langle m \rangle^A | y | \right)^{ n + 1 } } dy 
\\
&\sim 
\langle \ell \rangle ^{-N} 
\,
\| \sigma ; S^{0}_{\alpha, \alpha} \|_{N'}
\cdot 
\sup_{y\in\R^n}
	\frac{ \left|  \eta_m^\alpha(D) f (x-y) \right|  }{ 1 + \left( \langle m \rangle^A | y | \right)^{ n/r} } 
,
\end{align*}
where $N'= n + [n/r] + 2 + N \in \Z_+$.
Then, taking the $L^p$ (quasi)-norm of both sides
and applying Lemma \ref{maximal inequality for alpha modulation},
we obtain
\begin{equation*}
\left\| \widetilde \sigma_{\ell,m} (X,D) \eta_m^\alpha(D) f \right\|_{L^p } 
\lesssim 
\langle \ell \rangle ^{-N} 
\| \sigma ; S^{0}_{\alpha, \alpha} \|_{N'}
\cdot
\left\| \eta_m^\alpha(D) f  \right\|_{L^p } ,
\end{equation*}
which is the desired result.
\end{proof}



\section{Proof of the main theorems} \label{sec4}


In this section, we prove Theorem \ref{main theorem} and Corollaries \ref{main cor} and \ref{sharp cor}
stated in Section \ref{sec1}. 
We first give a proof of Theorem \ref{main theorem}.
To this end, we prepare two facts.
The first one is as follows.

\begin{lemma} \label{convolution 0<p<1 alpha modulation}
Let $ 0 < p \leq \infty$. Then we have
\begin{align*}
\left\| \varrho_k^\alpha(D) \left[ \widetilde \sigma_{\ell,m} (X,D) \eta_m^\alpha(D) f \right] \right \|_{L^p} \lesssim \langle \ell \rangle^{ A n ( \frac{ 1 }{ \min(1,p) } - 1 ) } \left\| \widetilde \sigma_{\ell,m} (X,D) \eta_m^\alpha(D) f  \right \|_{L^p}
\end{align*}
for all $f \in \calS (\R^n)$ and all $k,\ell,m \in \Z^n$.
\end{lemma}

\begin{proof}[Proof of Lemma \ref{convolution 0<p<1 alpha modulation}]
The case $1\leq p \leq \infty$ follows from the Young inequality. Assume that $0 < p < 1$. 
We recall from Proposition \ref{equivalent norm 0} and Lemma \ref{region} that
\begin{align*}
\supp \varrho_k^\alpha 
&\subset 
\left\{ \zeta \in \R^n :
\left| \zeta - \langle k \rangle^A k \right| \leq 2C\langle k \rangle^A 
\right\} ;
\\
\supp \calF [\sigma_{\ell,m} (X,D) f ] 
&\subset 
\left\{ \zeta \in \R^n : 
\left| \zeta - \langle m \rangle^A (\ell + m ) \right| \leq (C+ \sqrt{ n } )\langle m \rangle^A 
\right\},
\end{align*}
and from \eqref{change to tilde sigma} that
\[
[ \sigma_{\ell,m} (X,D) f ] (x) 
= [ \widetilde \sigma_{\ell,m} (X,D) \eta_m^\alpha(D) f ](x).
\]
Combining these, we see that $\varrho_k^\alpha(D) \left[ \widetilde \sigma_{\ell,m} (X,D) \eta_m^\alpha(D) f \right] $ always vanishes unless 
\begin{equation*}
\left| \langle m \rangle^A (\ell + m ) - \langle k \rangle^A k \right| \leq 2C\langle k \rangle^A + (C+ \sqrt{ n } )\langle m \rangle^A.
\end{equation*}
Hence, in those cases when 
$\varrho_k^\alpha(D) \left[ \widetilde \sigma_{\ell,m} (X,D) \eta_m^\alpha(D) f \right] $
does not vanish identically,
we obtain
\begin{align*}
\supp \calF[\sigma_{\ell,m} (X,D) f ] &\subset \left\{ \zeta \in \R^n : \left| \zeta - \langle k \rangle^A k \right| \leq 2C\langle k \rangle^A + 2(C+ \sqrt{ n } )\langle m \rangle^A \right\} ;\\
\supp \varrho_k^\alpha &\subset \left\{ \zeta \in \R^n : \left| \zeta - \langle k \rangle^A k \right| \leq 2C\langle k \rangle^A + 2(C+ \sqrt{ n } )\langle m \rangle^A \right\}.
\end{align*}
Moreover, recalling Lemma \ref{region} (or Remark \ref{cor region}), we have $| k - m | \lesssim \langle \ell \rangle$, which implies $\langle m \rangle^A \lesssim \langle k \rangle^A + \langle \ell \rangle^A$.
Hence, we have by Proposition \ref{convolution 0<p<1}
\begin{align*}
\left\| 
	\varrho_k^\alpha(D) 
		\left[ \widetilde \sigma_{\ell,m} (X,D) \eta_m^\alpha(D) f \right] 
\right \|_{L^p} 
&\lesssim 
\left( \langle k \rangle^A + \langle m \rangle^A \right)^{ n ( \frac{ 1 }{ p } - 1 ) } 
\left\| \calF^{-1} [\varrho_k^\alpha] \right\|_{L^p}
\cdot \left\| \widetilde \sigma_{\ell,m} (X,D) \eta_m^\alpha(D) f \right \|_{L^p} 
\\
&\lesssim 
\left( \langle k \rangle^A + \langle \ell \rangle^A \right)^{ n ( \frac{ 1 }{ p } - 1 ) } 
\langle k \rangle^{ - A n ( \frac{1}{p} - 1 ) } 
\cdot \left\| \widetilde \sigma_{\ell,m} (X,D) \eta_m^\alpha(D) f \right \|_{L^p} 
\\
&\lesssim 
\langle \ell \rangle^{ A n ( \frac{ 1 }{ p } - 1 ) } 
\left\| \widetilde \sigma_{\ell,m} (X,D) \eta_m^\alpha(D) f \right \|_{L^p},
\end{align*}
which completes the proof.
\end{proof}

As the second preparation for the proof of Theorem \ref{main theorem}, 
we note that 
\begin{align}
\label{triangle}
\begin{split}
\left\| \sum_{\ell\in\Z^n} f_{\ell} (x) \right\|_{L^p(\R^n_x)} 
&\leq 
\left( \sum_{\ell\in\Z^n} \left\| f_{\ell} (x) \right\|_{L^p(\R^n_x)}^{\min(1,p)} \right)^{\frac{1}{\min(1,p)}};
\\
\left\| \sum_{\ell\in\Z^n} a_{k,\ell} \right\|_{\ell^q(\Z^n_k)} 
&\leq 
\left( \sum_{\ell\in\Z^n} \left\|  a_{k,\ell} \right\|_{\ell^q(\Z^n_k)}^{\min(1,q)} \right)^{\frac{1}{\min(1,q)}}
\end{split}
\end{align}
hold for $0 < p,q \leq \infty$. 
For $p \geq 1$ or $q \geq 1$,
these are just the triangle inequality.
For $0 < p < 1$ or $ 0 < q < 1$,
these estimates follow from the fact that $| \sum a_k |^p \leq \sum | a_k |^p$,
i.e., the embedding $\ell^p \hookrightarrow \ell^1$. 

Now, we begin with the proof of Theorem \ref{main theorem}. 

\begin{proof}[Proof of Theorem \ref{main theorem}] 

Due to Proposition \ref{lift op} and the symbolic calculus (see Section \ref{sec2.1}), 
it suffices to prove Theorem \ref{main theorem} only for $s=0$. 
In fact, if $\sigma \in S^0_{\alpha, \alpha}$,
then there is a symbol $\tau \in S^0_{\alpha, \alpha}$ such that
$J^s \sigma (X,D) J^{-s} = \tau (X,D)$ and the estimate
$
\| \tau ; S^0_{\alpha, \alpha} \|_{L} \leq C \| \sigma ; S^0_{\alpha, \alpha} \|_{L'}
$
holds, where the constants $C$ and $L'$ depend on $L,s,n$.
Here, we set $J = (I-\Delta)^{1/2}$.
Hence, assuming that Theorem \ref{main theorem} 
(or more precisely \eqref{main theorem estimate}) with $s=0$ holds,
we have by Proposition \ref{lift op} for some $L,L'$
\begin{align*}
\left\| \sigma (X,D) f \right\|_{ M_{ p , q }^{ s , \alpha } } 
\sim 
\left\| J^s \sigma (X,D) J^{-s} J^s f \right\|_{ M_{ p , q }^{ 0 , \alpha } } 
&=
\left\| \tau (X,D) \left[ J^s f \right] \right\|_{ M_{ p , q }^{ 0 , \alpha } } 
\\
&\lesssim 
\| \tau ; S^0_{\alpha, \alpha} \|_{L} 
\left\| J^s f \right\|_{ M_{ p , q }^{ 0 , \alpha } } 
\lesssim 
\| \sigma ; S^0_{\alpha, \alpha} \|_{L'} 
\left\| f \right\|_{ M_{ p , q }^{ s , \alpha } }.
\end{align*}

We actually prove Theorem \ref{main theorem} for $s =0$.
We first estimate the $L^p$ (quasi)-norm of $\varrho_k^\alpha(D) [ \sigma(X,D) f ] $. 
Set $p^\ast = \min (1,p)$. 
Then, we have
\begin{align*}
\left\| \varrho_k^\alpha(D) \left[ \sigma(X,D) f  \right] \right\|_{L^p} 
&=
\left\| \sum_{ \ell\in \Z^n } \sum_{ \substack{ m \in \Z^n \\ | k - m | \lesssim \langle \ell \rangle } } 
\varrho_k^\alpha(D) \left[ \widetilde \sigma_{\ell,m} (X,D) \eta_m^\alpha(D) f \right] \right \|_{L^p} 
\\
&\leq
\left( 
\sum_{ \ell\in \Z^n } \sum_{ \substack{ m \in \Z^n \\ | k - m | \lesssim \langle \ell \rangle } } 
\left\| \varrho_k^\alpha(D) \left[ \widetilde \sigma_{\ell,m} (X,D) \eta_m^\alpha(D) f \right] \right \|_{L^p}^{p^\ast} 
\right)^{1/p^\ast} 
\\
&\lesssim
\left( \sum_{ \ell\in \Z^n } \sum_{ \substack{ \widetilde m \in \Z^n \\ | \widetilde m | \lesssim \langle \ell \rangle } } 
\langle \ell \rangle^{ \left\{ A n ( \frac{ 1 }{ p^\ast } - 1 ) - N \right\} \cdot p^\ast}  
\left\|\sigma ; S^0_{\alpha,\alpha}\right\|_{N'}^{p^\ast}
\left\| \eta_{k-\widetilde m}^\alpha(D) f \right \|_{L^p}^{p^\ast} 
\right)^{1/p^\ast}
\end{align*}
for all $k\in\Z^n$,
where $N' = N'_{n,p,N} \in \Z_+$ is the constant given in Lemma \ref{L^p to L^p}.
Also, we applied \eqref{teinei}, \eqref{change to tilde sigma} and Lemma \ref{region} 
to obtain the first line. 
In the second line, we used the first estimate in \eqref{triangle}.
In the last line, 
we invoked Lemmas \ref{L^p to L^p} and \ref{convolution 0<p<1 alpha modulation}. 

Next, recalling Proposition \ref{equivalent norm 0} 
and taking the $\ell^q (\Z^n_k)$ (quasi)-norm of the above estimate, 
we obtain
\begin{align*}
\left\| \sigma(X,D) f \right\|_{M^{0,\alpha}_{p,q}} 
&\lesssim 
\left\|\sigma ; S^0_{\alpha,\alpha}\right\|_{N'}
\left\| 
\sum_{ \ell\in \Z^n } \sum_{ \substack{ \widetilde m \in \Z^n \\ | \widetilde m | \lesssim \langle \ell \rangle } } 
	\langle \ell \rangle^{ \left\{ A n ( \frac{ 1 }{ p^\ast } - 1 ) - N \right\}\cdot p^\ast}
	\left\| \eta_{k-\widetilde m}^\alpha(D) f \right \|_{L^p}^{p^\ast} 
\right\|_{\ell^{q/p^\ast} (\Z^n_k)}^{1/p^\ast}
\\
&\leq 
\left\|\sigma ; S^0_{\alpha,\alpha}\right\|_{N'}
\left(
\sum_{ \ell\in \Z^n } \sum_{ \substack{ \widetilde m \in \Z^n \\ | \widetilde m | \lesssim \langle \ell \rangle } } 
	\langle \ell \rangle^{ \left\{ A n ( \frac{ 1 }{ p^\ast } - 1 ) - N \right\} \cdot \min(1,p,q)} 
	\Big\| \left\| \eta_{k-\widetilde m}^\alpha(D) f \right \|_{L^p} \Big\|_{\ell^{q}(\Z^n_k)}^{\min (1,p,q) } 
\right)^{\frac{1}{\min(1, p, q)}} 
\\
&\lesssim 
\left\|\sigma ; S^0_{\alpha,\alpha}\right\|_{N'}
\left\| f \right\|_{M_{p,q}^{0,\alpha} } 
\left(\sum_{ \ell\in \Z^n } 
	\langle \ell \rangle^{ \left\{ A n ( \frac{ 1 }{ p^\ast } - 1 ) - N \right\} \cdot \min(1,p,q) + n }
\right)^{\frac{1}{\min(1, p, q)}},
\end{align*}
where in the second inequality we used the second inequality in \eqref{triangle} and the identity 
\[
p^\ast \cdot \left( \frac{q}{p^\ast} \right)^\ast 
= p^\ast \cdot \min \left(1, \frac{q}{p^\ast} \right) = \min (p^\ast, q) = \min (1,p,q).
\]
Therefore, choosing $N \in \Z_+$ such that 
$ \left\{ A n ( \frac{ 1 }{ p^\ast } - 1 ) - N \right\} \cdot \min(1,p,q) + n < -n$, 
then we have
\[
\left\| \sigma(X,D) f \right\|_{M^{0,\alpha}_{p,q}} 
\lesssim 
\left\|\sigma ; S^0_{\alpha,\alpha}\right\|_{N'}
\| f \|_{M_{p,q}^{0,\alpha} },
\] 
which completes the proof of the main theorem.
\end{proof}

Next, we prove Corollary \ref{main cor}. 
This is immediately given from Theorem \ref{main theorem} and the symbolic calculus.
\begin{proof}[Proof of Corollary \ref{main cor}] 
We observe that if $\sigma \in S_{\rho, \delta}^b$, then $\sigma \in S_{\alpha, \alpha}^b$, 
since $0 \leq \delta \leq \alpha \leq \rho$.
Here, recall that $0\leq \alpha <1$.
Then, the symbolic calculus shows that there is a symbol $\tau \in S_{\alpha, \alpha}^0$
such that $J^{-b} \sigma(X,D) = \tau(X,D)$,
where $J = (I-\Delta)^{1/2}$.
Thus, Theorem \ref{main theorem} shows that 
$\| \tau(X,D) f \|_{M_{p,q}^{s,\alpha}} \lesssim \| f \|_{M_{p,q}^{s,\alpha}}$
for all $f \in \calS(\R^n)$.
In combination with Proposition \ref{lift op}, this implies that
\begin{equation*}
\| \sigma(X,D) f \|_{M_{p,q}^{s-b,\alpha}}
\sim \| J^{-b} \sigma(X,D) f \|_{M_{p,q}^{s,\alpha}}
= \| \tau(X,D) f \|_{M_{p,q}^{s,\alpha}}
\lesssim \| f \|_{M_{p,q}^{s,\alpha}}.
\end{equation*}
This completes the proof.
\end{proof}

We next prove Corollary \ref{sharp cor}. 
In order to achieve this, 
we recall the following inclusion relations 
between modulation spaces and $\alpha$-modulation spaces
given by \cite[Theorem 4.1]{han wang 2014} and \cite[Section 1]{toft wahlberg 2012}.

\begin{lemma} 
\label{embedding 0-1}
Let $0 < q \leq \infty$ and $0 \leq \alpha < 1$. 

\noindent{\rm (1)} 
$M_{2,q}^{s_1,\alpha} (\R^n) \subset M_{2,q}^{0} (\R^n)$ holds 
for $s_1 =  n \alpha \cdot \max(0, 1/q-1/2)$;

\noindent{\rm (2)} 
$M_{2,q}^{0} (\R^n) \subset M_{2,q}^{s_2,\alpha} (\R^n)$ holds 
for $s_2 =  n \alpha \cdot  \min(0, 1/q-1/2)$.
\end{lemma}

Now, let us start the proof.

\begin{proof}[Proof of Corollary \ref{sharp cor}] 
The ``IF'' part immediately follows from the relation 
$S^{0}_{\rho, \delta} \subset S^{0}_{\alpha, \alpha}$ for $\delta \leq \alpha \leq \rho$ 
and Theorem \ref{main theorem}, 
so that we only consider the ``ONLY IF'' part. 
We assume that all $\sigma(X,D) \in \op (S_{\rho, \delta}^0)$ 
are bounded on $M_{2,q}^{s,\alpha} (\R^n)$.
Then all $\widetilde{\sigma}(X,D) \in \op (S_{\rho, \delta}^{-s_1+s_2})$ 
are also bounded on $M^0_{2,q} (\R^n)$,
where $s_1,s_2$ are as in Lemma \ref{embedding 0-1}.
Indeed, the symbolic calculus shows that
$J^{- s + s_1} \widetilde \sigma (X,D) J^{s - s_2}
\in \op(S_{\rho, \delta}^0)$
for $\widetilde{\sigma} \in S_{\rho, \delta}^{-s_1+s_2}$,
where $J = (I-\Delta)^{1/2}$.
Therefore, we have by Proposition \ref{lift op} and Lemma \ref{embedding 0-1}
\begin{align*}
\left\| \widetilde \sigma (X,D) f \right\|_{M_{2,q}^{0}} 
&\lesssim
\left\| \widetilde \sigma (X,D) f \right\|_{M_{2,q}^{s_1,\alpha}} 
\lesssim 
\left\| J^{- s + s_1} \widetilde \sigma (X,D) f \right\|_{M_{2,q}^{s, \alpha}} 
\\
&= 
\left\| J^{- s + s_1} \widetilde \sigma (X,D) J^{s - s_2} J^{-s + s_2} f \right\|_{M_{2,q}^{s, \alpha}} 
\lesssim \left\| J^{ - s + s_2}  f \right\|_{M_{2,q}^{s, \alpha}} 
\sim \left\| f \right\|_{M_{2,q}^{s_2, \alpha}} 
\lesssim \left\| f \right\|_{M_{2,q}^{0}} .
\end{align*}
This yields that
$\op (S_{\rho, \delta}^{-s_1+s_2}) \subset \mathcal{L} \big(M_{2,q}^{0} (\R^n) \big)$,
and thus Theorem A gives $- s_1 + s_2 \leq - | 1/q - 1/2 | \delta n$.
Here, since $-s_1+s_2 = -n\alpha|1/q-1/2|$,
this is equivalent to $0 \leq - | 1/q - 1/2 | ( \delta - \alpha ) n $. 
Hence, because of $q \neq 2$, we obtain $\delta \leq \alpha$, 
which concludes the proof of ``ONLY IF'' part in Corollary \ref{sharp cor}.
\end{proof}

\begin{remark} \label{alpha -}
In this remark, we find a counterexample to the inclusion 
$
\op (S_{\alpha - \varepsilon, \alpha - \varepsilon}^0) 
\subset 
\mathcal{L}(M_{p,q}^{s,\alpha})
$ 
for $0 < \varepsilon < \alpha$ and $0 < p < 1$. 
We write
$A_\varepsilon = \frac{\alpha - \varepsilon}{ 1 - \alpha}$ 
for $0 < \varepsilon < \alpha$. 
Choose $\psi , \widetilde \psi \in \calS( \R^n ) $ satisfying that 
$\supp \psi \subset \{ \xi \in \R^n : | \xi | \leq c \}$, 
$\widetilde \psi (\xi) = 1$ on $\{ \xi \in \R^n : | \xi | \leq c \}$ 
and $\supp \widetilde \psi \subset \{ \xi \in \R^n : | \xi | \leq 2c \}$. 
Here, the constant $c=c_\alpha > 0$ is so small that the sets $B_k$, $k\in\Z^n$, are pairwise disjoint,
where 
$B_k = \{ \xi \in \R^n: | \xi-\langle k \rangle^A k | \leq 2 c \langle k \rangle^A \}$.
We will see in Remark \ref{small c} that such a constant $c>0$ exists. 
Set
\begin{equation*}
\sigma(\xi) = \sum_{m \in \Z^n} \psi \left( \frac{ \xi-\langle m \rangle^A m}{ \langle m \rangle^{A_\varepsilon} } \right) 
\aand
\widehat{ f_{\ell} } (\xi) = \widetilde \psi \left( \frac{ \xi - \langle \ell \rangle^A \ell}{ \langle \ell \rangle^{A} } \right)
\end{equation*}
for all $\ell \in \Z^n$. 
Here, $\psi ( \frac{ \cdot - \langle m \rangle^A m}{ \langle m \rangle^{A_\varepsilon} } ) \subset B_m$,
since $A > A_\varepsilon $.
Then, it follows that
\begin{equation} \label{psi2}
\supp \psi \left( \frac{ \cdot - \langle m \rangle^A m}{ \langle m \rangle^{A_\varepsilon} } \right) 
\cap \supp \widetilde\psi \left( \frac{ \cdot - \langle \ell \rangle^A \ell }{ \langle \ell \rangle^{A} } \right) 
\subset B_m \cap B_\ell
= \varnothing 
\iif 
m \neq \ell
\end{equation}
and
\begin{equation} \label{psi3}
\widetilde \psi \left( \frac{ \xi - \langle m \rangle^A m}{ \langle m \rangle^{A} } \right)  = 1 
\textrm{ on } 
\supp \psi \left( \frac{ \cdot - \langle m \rangle^A m}{ \langle m \rangle^{A_\varepsilon} } \right) .
\end{equation} 
In addition, note that at most one term in the sum defining $\sigma$ is non-zero for each $\xi$, 
since 
\begin{equation}
\label{psi4}
\supp \psi \left( \frac{ \cdot - \langle m \rangle^A m}{ \langle m \rangle^{A_\varepsilon} } \right) 
\cap 
\supp \psi \left( \frac{ \cdot - \langle m^\prime \rangle^A m^\prime }{ \langle m^\prime \rangle^{A_\varepsilon} } \right)
\subset B_m \cap B_{m'}
= \varnothing 
\iif
m \neq m^\prime.
\end{equation}
Then, since $\langle \xi \rangle \sim \langle m \rangle^{ \frac{ 1 }{ 1 - \alpha } }$ 
if $\xi \in \supp \psi \left( \frac{ \cdot - \langle m \rangle^A m}{ \langle m \rangle^{A_\varepsilon} } \right)$
and $\sigma$ is the $x$-independent symbol,
we see that $\sigma \in S_{\alpha - \varepsilon, \delta}^0$ 
for any $0 \leq \delta \leq 1$. 
In particular, $\sigma \in S_{\alpha - \varepsilon, \alpha - \varepsilon}^0$. 
Now, by using these functions $\sigma$ and $f_\ell$, 
we will prove that 
$\op (S_{\alpha - \varepsilon, \alpha - \varepsilon}^0) \subset \mathcal{L}(M_{p,q}^{s,\alpha})$ 
does not hold for $0 < p < 1$. 

We first estimate the $\alpha$-modulation space quasi-norm of $f_\ell$. 
By Proposition \ref{equivalent norm 0}, we have
\begin{align*}
\| f_\ell \|_{M_{p,q}^{s,\alpha}} 
&\sim
\left\| \langle k \rangle^{s/(1-\alpha)} \left\| \varrho_k^\alpha(D) f_\ell \right\|_{L^p} \right\|_{\ell^q (\Z^n_k)}
\\
&=
\Bigg\|
\langle k \rangle^{\frac{s}{1-\alpha}} 
\left\| \left (\calF^{-1} \left[ \varrho_k^\alpha \right] \right)
\ast 
\left(\calF^{-1} 
\left[ \widetilde \psi \left( \frac{ \cdot - \langle \ell \rangle^A \ell}{ \langle \ell \rangle^{A} } \right)\right] 
\right) \right\|_{L^p} \Bigg\|_{\ell^q (\Z^n_k:| k - \ell | \lesssim 1)}
\end{align*}
for all $\ell \in \Z^n$. 
Here, the summation in $\ell^q(\Z^n_k)$ is restricted to $| k - \ell | \lesssim 1$
(otherwise, $\varrho_k^\alpha(D) f_\ell$ vanishes).
This restriction is due to the relation:
\begin{equation}
\label{fellrhok}
|\langle k \rangle^A k - \langle \ell \rangle^A \ell | \lesssim \langle k \rangle^A+\langle \ell \rangle^A,
\end{equation}
which is obtained from the information of $\supp\varrho_k^\alpha$ and $\supp \widehat{f_\ell}$:
\begin{equation*}
|\xi-\langle k \rangle^A k| \lesssim \langle k \rangle^A
\aand
|\xi-\langle \ell \rangle^A \ell| \lesssim \langle \ell \rangle^A,
\end{equation*}
and Lemma \ref{k and m}. 
Then, since this restriction leads $\langle \ell \rangle \sim \langle k \rangle$, it follows that
\[
\supp \varrho_k^\alpha
\subset
\left\{ \xi \in \R^n : | \xi - \langle k \rangle^A k | \lesssim \langle k \rangle^A \right\}
\subset
\left\{ \xi \in \R^n : | \xi - \langle \ell \rangle^A \ell | \lesssim \langle \ell \rangle^A \right\}.
\]
This is because we have by \eqref{fellrhok}
\begin{align*}
| \xi - \langle k \rangle^A k | \lesssim \langle k \rangle^A
&\Longrightarrow
| \xi - \langle \ell \rangle^A \ell | \lesssim \langle k \rangle^A +| \langle k \rangle^A k - \langle \ell \rangle^A \ell |
\\&\Longrightarrow
| \xi - \langle \ell \rangle^A \ell | \lesssim \langle \ell \rangle^A.
\end{align*}
Hence, Proposition \ref{convolution 0<p<1} gives that
\begin{equation*}
\| f_\ell \|_{M_{p,q}^{s,\alpha}} 
\lesssim
\Bigg\|
\langle \ell \rangle^{\frac{s}{1-\alpha} + An(\frac{1}{p} - 1)}
\left\| \calF^{-1} \left[ \varrho_k^\alpha \right]\right\|_{L^p} 
\cdot
\left\| \calF^{-1}
\left[ \widetilde \psi \left( \frac{ \cdot - \langle \ell \rangle^A \ell}{ \langle \ell \rangle^{A} } \right)\right] 
\right\|_{L^p} \Bigg\|_{\ell^q (\Z^n_k:| k - \ell | \lesssim 1)}.
\end{equation*}
Performing the changes of the variables for both $L^p$ (quasi)-norms
and using the fact that $\langle \ell \rangle \sim \langle k \rangle$,
we conclude that
\begin{align*}
\| f_\ell \|_{M_{p,q}^{s,\alpha}} 
\lesssim 
\langle \ell \rangle^{\frac{s}{(1-\alpha)}} \cdot \langle \ell \rangle^{An(1 - \frac{1}{p})} 
\left\|
1
\right\|_{\ell^q (\Z^n_k:| k - \ell | \lesssim 1)}
\sim 
\langle \ell \rangle^{\frac{s}{1-\alpha}} \cdot \langle \ell \rangle^{An(1 - \frac{1}{p})} 
\end{align*}
for all $\ell \in \Z^n$.

We next consider the $\alpha$-modulation space quasi-norm of $\sigma (X,D) f_\ell$. 
Using \eqref{psi2}, \eqref{psi3} and Proposition \ref{equivalent norm 0}, 
we have
\begin{align*}
\| \sigma(X,D) f_\ell \|_{M_{p,q}^{s,\alpha}} 
&\sim 
\left\| \langle k \rangle^{\frac{s}{1-\alpha}} \left\| \calF^{-1} 
\left[ \varrho_k^\alpha \cdot 
	\left( \sum_{m \in \Z^n} \psi \left( \frac{ \cdot - \langle m \rangle^A m}{ \langle m \rangle^{A_\varepsilon} } \right) \right ) 
\cdot 
\widetilde \psi \left( \frac{ \cdot - \langle \ell \rangle^A \ell}{ \langle \ell \rangle^{A} } \right) 
\right] 
\right\|_{L^p} \right\|_{\ell^q (\Z^n_k)}
\\
&\geq
\langle \ell \rangle^{\frac{s}{1-\alpha}} 
\left\| \calF^{-1} \left[ \varrho_\ell^\alpha 
\cdot 
\psi \left( \frac{ \cdot - \langle \ell \rangle^A \ell}{ \langle \ell \rangle^{A_\varepsilon} } \right) \right] 
\right\|_{L^p} .
\end{align*}
If we recall the definition of the function $\varrho$ in Proposition \ref{equivalent norm 0}, 
we see that $\varrho_\ell^\alpha (\xi) = 1 $ 
on the support of 
$\psi \left( \frac{ \cdot - \langle \ell \rangle^A \ell}{ \langle \ell \rangle^{A_\varepsilon} } \right)$
(possibly after shrinking the constant $c=c_\alpha$ further). 
Hence, we obtain
\begin{align*}
\| \sigma(X,D) f_\ell \|_{M_{p,q}^{s,\alpha}} 
\gtrsim 
\langle \ell \rangle^{\frac{s}{1-\alpha}} 
\left\| \calF^{-1} 
\left[ \psi \left( \frac{ \cdot - \langle \ell \rangle^A \ell}{ \langle \ell \rangle^{A_\varepsilon} } \right) \right] 
\right\|_{L^p} 
\sim 
\langle \ell \rangle^{\frac{s}{1-\alpha}} \cdot \langle \ell \rangle^{A_\varepsilon n ( 1 - \frac{1 }{p} ) }.
\end{align*}

We are now in position to prove the conclusion of this remark. 
We assume toward a contradiction that 
$\sigma (X,D)$ is bounded on $M_{p,q}^{s,\alpha}$. 
Then, we have
\begin{equation*}
\langle \ell \rangle^{\frac{s}{1-\alpha}} \cdot\langle \ell \rangle^{A_\varepsilon n (1-\frac{1}{p})} \lesssim \| \sigma(X,D) f_\ell \|_{M_{p,q}^{s,\alpha}} \lesssim \|  f_\ell \|_{M_{p,q}^{s,\alpha}} \lesssim \langle \ell \rangle^{\frac{s}{1-\alpha}} \cdot\langle \ell \rangle^{An (1-\frac{1}{p})} 
\end{equation*}
for all $\ell \in \Z^n$. However, since $A_\varepsilon < A$ and $0 < p < 1$, this is a contradiction. Therefore, $\sigma$ belongs to $S_{\alpha - \varepsilon, \alpha - \varepsilon}^0$, but $\sigma(X,D)$ is not bounded on $M_{p,q}^{s,\alpha}$. 
\end{remark}

\begin{remark} \label{small c}
We determine the detail quantity of the small constant $c=c_\alpha > 0$
which was used to ensure that the sets $B_{k}$ in Remark \ref{alpha -} are pairwise disjoint.
Assume that $B_m \cap B_\ell \neq \varnothing$,
which implies
\[
| \langle \ell \rangle^A \ell  - \langle m \rangle^A m |
\leq 2c (\langle \ell \rangle ^A + \langle m \rangle^A ) .
\]
We here recall from Lemma \ref{k and m} that
\[
(\langle \ell \rangle ^A + \langle m \rangle^A ) | \ell - m | 
\leq K \left| \langle \ell \rangle^A \ell  - \langle m \rangle^A m \right|
\]
holds for all $m,\ell \in \Z^n$, where the constant $K = K_\alpha > 0$ 
(careful reading gives that this constant is $\max (6, 1 + 2^A)$ at most).
Then, we have
\[
| m - \ell | 
\leq K 
\left| \langle \ell \rangle^A \ell  - \langle m \rangle^A m \right| 
/ (\langle \ell \rangle ^A + \langle m \rangle^A ) 
\leq 2cK.
\]
Choosing the constant $c$ satisfying $2cK < 1$,
we have $|m-\ell|<1$, i.e., $m=\ell$.
\end{remark}



\section*{Acknowledgments}
The authors sincerely express deep gratitude to the anonymous referees 
for their careful reading and giving fruitful suggestions and comments.
The first author is supported by Grant-in-Aid for JSPS Research Fellow (No. 17J00359). 
The second author is partially supported by Grant-in-aid for Scientific Research from JSPS (No. 16K05201).




\end{document}